\documentclass[12pt]{article}

\usepackage{amsmath,amsthm, amssymb, color, enumerate, fullpage, graphicx}

\newtheorem{theorem}{Theorem}
\newtheorem{lemma}{Lemma}

\title{On a problem of Neumann}
\author{Michael Tait\thanks{Department of Mathematics, University of California San Diego. \texttt{mtait@math.ucsd.edu}}}
\date{}

\begin{document}

\maketitle
\abstract{A conjecture widely attributed to Neumann is that all finite non-desarguesian projective planes contain a Fano subplane. In this note, we show that any finite projective plane of even order which admits an orthogonal polarity contains a Fano subplane. The number of planes of order less than $n$ previously known to contain a Fano subplane was $O(\log n)$, whereas the number of planes of order less than $n$ that our theorem applies to is not bounded above by any polynomial in $n$.}

\section{Introduction}
A fundamental question in incidence geometry is about the subplane structure of projective planes.  There are relatively few results concerning when a projective plane of order $k$ is a subplane of a projective plane of order $n$. Neumann \cite{Neumann1954} found Fano subplanes in certain Hall planes, which led to the conjecture that every finite non-desarguesian plane contains $PG(2,2)$ as a subplane (this conjecture is widely attributed to Neumann, though it does not appear in her work). 

Johnson \cite{Johnson2009} and Fisher and Johnson \cite{FisherJohnson2010} showed the existence of Fano subplanes in many translation planes. Petrak \cite{Petrak2010} showed that Figueroa planes contain $PG(2,2)$ and Caliskan and Petrak \cite{CaliskanPetrak2013} showed that Figueroa planes of odd order contain $PG(2,3)$. Caliskan and Moorhouse \cite{CaliskanMoorhouse2011} showed that all Hughes planes contain $PG(2,2)$ and that the Hughes plane of order $q^2$ contains $PG(2,3)$ if $q\equiv 5 \pmod{6}$. We prove the following.

\begin{theorem}\label{main theorem}
Let $\Pi$ be a finite projective plane of even order which admits an orthogonal polarity. Then $\Pi$ contains a Fano subplane.
\end{theorem}

Ganley \cite{Ganley1972} showed that a finite semifield plane admits an orthogonal polarity if and only if it can be coordinatized by a commutative semifield. A result of Kantor \cite{Kantor2003} implies that the number of nonisomorphic planes of order $n$ a power of $2$ that can be coordinatized by a commutative semifield is not bounded above by any polynomial in $n$. Thus, Theorem \ref{main theorem} applies to many projective planes.

\section{Proof of Theorem \ref{main theorem}}
The proof of Theorem \ref{main theorem} is graph theoretic, and we collect some definitions and results first. Let $\Pi = (\mathcal{P}, \mathcal{L}, \mathcal{I})$ be a projective plane of order $n$. We write $p\in l$ or say $p$ is on $l$ if $(p,l)\in \mathcal{I}$. Let $\pi$ be a polarity of $\Pi$. That is, $\pi$ maps points to lines and lines to points, $\pi^2$ is the identity function, and $\pi$ respects incidence. Then one may construct the polarity graph $G^o_\pi$ as follows. $V(G^o_\pi) = \mathcal{P}$ and $p\sim q$ if and only if $p \in \pi(q)$. That is, the neighborhood of a vertex $p$ is the line $\pi(p)$ that $p$ gets mapped to under the polarity. If $p \in \pi(p)$, then $p$ is an {\em absolute point} and the vertex $p$ will have a loop on it. A polarity is {\em orthogonal} if exactly $n+1$ points are absolute. We note that as neighborhoods in the graph represent lines in the geometry, each vertex in $G^o_\pi$ has exactly $n+1$ neighbors (if $v$ is an absolute point, it has exactly $n$ neighbors other than itself). We provide proofs of the following preliminary observations for completeness. 

\begin{lemma}\label{observations}
Let $\Pi$ be a projective plane with polarity $\pi$, and $G^o_\pi$ be the associated polarity graph.
\begin{enumerate}[(a)]
\item For all $u,v\in V(G^o_\pi)$, $u$ and $v$ have exactly $1$ common neighbor. \label{partone}
\item $G^o_\pi$ is $C_4$ free. \label{parttwo}
\item If $u$ and $v$ are two absolute points of $G^o_\pi$, then $u\not\sim v$. \label{partthree}
\item If $v\in V(G^o_\pi)$, then the neighborhood of $v$ induces a graph of maximum degree at most $1$. \label{partfour}
\item Let $e = uv$ be an edge of $G^o_\pi$ such that neither $u$ nor $v$ is an absolute point. Then $e$ lies in a unique triangle in $G^o_\pi$. \label{partfive}
\end{enumerate}
\end{lemma}

\begin{proof}
To prove (\ref{partone}), let $u$ and $v$ be an arbitrary pair of vertices in $V(G^o_\pi)$. Because $\Pi$ is a projective plane, $\pi(u)$ and $\pi(v)$ meet in a unique point. This point is the unique vertex in the intersection of the neighborhood of $u$ and the neighborhood of $v$. (\ref{parttwo}) and (\ref{partthree}) follow from (\ref{partone}).

To prove (\ref{partfour}), if there is a vertex of degree at least $2$ in the graph induced by the neighborhood of $v$, then $G^o_\pi$ contains a $4$-cycle, a contradiction by (\ref{parttwo}).

Finally, let $u\sim v$ and neither $u$ nor $v$ an absolute point. Then by (\ref{partone}) there is a unique vertex $w$ adjacent to both $u$ and $v$. Now $uvw$ is the purported triangle, proving (\ref{partfive}).
\end{proof}

\begin{proof}[Proof of Theorem \ref{main theorem}]
We will now assume $\Pi$ is a projective plane of even order $n$, that $\pi$ is an orthogonal polarity, and that $G^o_\pi$ is the corresponding polarity graph (including loops). Since $n$ is even and $\pi$ is orthogonal, a classical theorem of Baer (\cite{Baer1946}, see also Theorem 12.6 in \cite{HughesPiper}) says that the $n+1$ absolute points under $\pi$ all lie on one line. Let $a_1,\ldots, a_{n+1}$ be the set of absolute points and let $l$ be the line containing them. Then there is some $p \in \mathcal{P}$ such that $\pi(l) = p$. This means that in $G^o_\pi$, the neighborhood of $p$ is exactly the set of points $\{a_1,\ldots, a_{n+1}\}$. For $1\leq i \leq n+1$, let $N_i$ be the neighborhood of $a_i$. Then by Lemma \ref{observations}.\ref{parttwo}, $N_i\cap N_j = \emptyset$ if $i\not= j$. Further, counting gives that 
\begin{equation}\label{vertexSet}
V(G^o_\pi) = p \cup \left(\bigcup_{i=1}^{n+1} a_i\right) \cup \left( \bigcup_{i=1}^{n+1} N_i \right).
\end{equation}
Let $ER^o_2$ be the graph on $7$ points which is the polarity graph (with loops) of $PG(2,2)$ under the orthogonal polarity. 

\begin{center}
\begin{figure}\caption{$ER^o_2$}
\centering
\includegraphics[height=6cm]{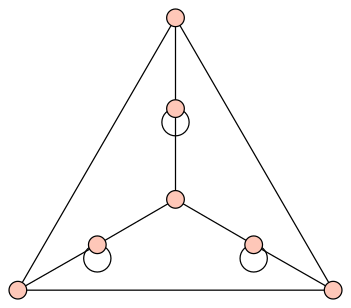}
\end{figure}
\end{center}

\begin{lemma}
If $ER^o_2$ is a subgraph of $G^o_\pi$, then $\Pi$ contains a Fano subplane.
\end{lemma}
\begin{proof}
Let $v_1,\ldots, v_7$ be the vertices of a subgraph $ER^o_2$ of $G^o_\pi$. Let $l_i = \pi(v_i)$ for $1\leq i\leq 7$. Then the lines $l_1 ,\ldots , l_7$ in $\Pi$ restricted to the points $v_1,\ldots , v_7$ form a point-line incidence structure, and one can check directly that it satisfies the axioms of a projective plane.
\end{proof}
Thus, it suffices to find $ER^o_2$ in $G^o_\pi$. To find $ER^o_2$ it suffices to find distinct $i,j,k$ such that there are $v_i\in N_i$, $v_j\in N_j$, and $v_k\in N_k$ where $v_iv_jv_k$ forms a triangle in $G^o_\pi$, for then the points $p, a_i, a_j, a_k, v_i, v_j, v_k$ yield the subgraph $ER^o_2$. Now note that for all $i$, and for $v\in N_i$, $v$ has exactly $n$ neighbors that are not absolute points. There are $n+1$ choices for $i$ and $n-1$ choices for $v\in N_i$. As each edge is counted twice, this yields
\[
\frac{n(n-1)(n+1)}{2}
\]
edges with neither end an absolute point. By Lemma \ref{observations}.\ref{partfive}, there are at least 
\[
\frac{n^3-n}{6}
\]
triangles in $G^o_\pi$. By Lemma \ref{observations}.\ref{partthree}, there are no triangles incident with $p$, by Lemma \ref{observations}.\ref{parttwo}, there are no triangles that have more than one vertex in $N_i$ for any $i$, and by Lemma \ref{observations}.\ref{partfour} there are at most $\lfloor\frac{n-1}{2}\rfloor = \frac{n}{2} - 1$ triangles incident with $a_i$ for each $i$. Therefore, by \eqref{vertexSet}, there are at least 
\[
\frac{n^3 - n}{6} - (n+1)\left(\frac{n}{2}-1\right)
\]
copies of $ER^o_2$ in $G^o_\pi$. This expression is positive for all even natural numbers $n$.
\end{proof}

\section{Concluding Remarks}
First, we note that the proof of Theorem \ref{main theorem} actually implies that there are $\Omega\left(n^3\right)$ copies of $PG(2,2)$ in any plane satisfying the hypotheses, and echoing Petrak \cite{Petrak2010}, perhaps one could find subplanes of order $4$ for $n$ large enough. We also note that it is crucial in the proof that the absolute points form a line. When $n$ is odd, the proof fails (as it must, since our proof does not detect if $\Pi$ is desarguesian or not).

\section*{Acknowledgments}
The author would like to thank Gary Ebert and Eric Moorhouse for helpful comments.

\bibliographystyle{plain}
\bibliography{bib}

\begin{thebibliography}{10}

\bibitem{Baer1946}
Reinhold Baer.
\newblock Projectivities with fixed points on every line of the plane.
\newblock {\em Bulletin of the American Mathematical Society}, 52(4):273--286,
  1946.

\bibitem{CaliskanMoorhouse2011}
Cafer Caliskan and G~Eric Moorhouse.
\newblock Subplanes of order 3 in hughes planes.
\newblock {\em The Electronic Journal of Combinatorics}, 18(P2):1, 2011.

\bibitem{CaliskanPetrak2013}
Cafer Caliskan and Bryan Petrak.
\newblock Subplanes of order 3 in figueroa planes.
\newblock {\em Finite Fields and Their Applications}, 20:24--29, 2013.

\bibitem{FisherJohnson2010}
J~Chris Fisher and Norman~L Johnson.
\newblock Fano configurations in subregular planes.
\newblock {\em Note di Matematica}, 28(2):69--98, 2010.

\bibitem{Ganley1972}
MJ~Ganley.
\newblock Polarities in translation planes.
\newblock {\em Geometriae Dedicata}, 1(1):103--116, 1972.

\bibitem{HughesPiper}
Daniel~R Hughes and Frederick~Charles Piper.
\newblock {\em Projective planes}, volume~6.
\newblock Springer, 1973.

\bibitem{Johnson2009}
Norman~L Johnson.
\newblock Fano configurations in translation planes of large dimension.
\newblock {\em Note di Matematica}, 27(1):21--38, 2009.

\bibitem{Kantor2003}
William~M Kantor.
\newblock Commutative semifields and symplectic spreads.
\newblock {\em Journal of Algebra}, 270(1):96--114, 2003.

\bibitem{Neumann1954}
Hanna Neumann.
\newblock On some finite non-desarguesian planes.
\newblock {\em Archiv der Mathematik}, 6(1):36--40, 1954.

\bibitem{Petrak2010}
Bryan Petrak.
\newblock Fano subplanes in finite figueroa planes.
\newblock {\em Journal of Geometry}, 99(1-2):101--106, 2010.

\end{thebibliography}

\end{document}